\documentclass[12pt]{amsart}
\usepackage{amscd,amssymb,verbatim}
\usepackage{amssymb,amsmath,amsthm,epsfig}

\usepackage{color}

  \theoremstyle{plain}
  \newtheorem{thm}{Theorem}[section]
  \theoremstyle{definition}
  
  \theoremstyle{plain}
  \newtheorem{lem}[thm]{Lemma}
  \theoremstyle{plain}
  \newtheorem{cor}[thm]{Corollary}
  \theoremstyle{plain}
  \newtheorem{prop}[thm]{Proposition}
  \theoremstyle{remark}
  \newtheorem*{claim*}{Claim}
  
  \theoremstyle{definition}
  
  \newtheorem{probs}[thm]{Problems}
\newcommand{\N}{\mathbb{N}}

\newcommand{\Z}{\mathbb{Z}}
\newcommand{\R}{\mathbb{R}}
\newcommand{\E}{\mathbb{E}}

\def\Span{\operatorname{span}}
\def\supp{\operatorname{supp}}

\newcommand{\vp}{\varepsilon}

\def\gb{\overline{g}}
\allowdisplaybreaks
\begin{document}

\title{Unconditional structures of translates for $L_p(\R^d)$}
\author{D. Freeman}
\address{Department of Mathematics and Computer Science\\
St Louis University\\
St Louis, MO 63103  USA} \email{dfreema7@slu.edu}
\author{E. Odell}
\address{Department of Mathematics \\
The University of Texas\\ 1 University Station C1200\\
Austin, TX 78712  USA} \email{odell@math.utexas.edu}
\author{Th. Schlumprecht}
\address{Department of Mathematics, Texas A\&M University\\
College Station, TX 77843, USA}
\email{thomas.schlumprecht@math.tamu.edu}

\author{A. Zs\'ak}
\address{Peterhouse, Cambridge, CB2 1RD, UK}
\email{A.Zsak@dpmms.cam.ac.uk}

\thanks{Research of the first, second, and third author was supported by the
  National Science Foundation.
} \subjclass[2000]{46B20, 54H05, 42C15}

\maketitle

\begin{abstract}

We prove that a sequence $(f_i)_{i=1}^\infty$ of translates of a
fixed $f\in L_p(\R)$ cannot be an unconditional basis of $L_p(\R)$
for any $1\le p<\infty$. In contrast to this, for every
$2<p<\infty$, $d\in\N$ and unbounded sequence
$(\lambda_n)_{n\in\N}\subset \R^d$ we establish the existence of a
function $f\in L_p(\R^d)$ and sequence $(g^*_n)_{n\in\N}\subset
L_p^*(\R^d)$ such that $(T_{\lambda_n} f, g^*_n)_{n\in\N}$ forms an
unconditional Schauder frame for $L_p(\R^d)$. In particular, there
exists a Schauder frame of integer translates for $L_p(\R)$ if (and
only if) $2<p<\infty$.

\end{abstract}

\section{Introduction}

If $d\in\N$ and $\lambda\in\R^d$, the {\em translation operator}
$T_\lambda$ is defined by $T_\lambda f(x)=f(x-\lambda)$ for all
$x\in \R^d$ and $f:\R^d\rightarrow\R^d$.  Note that for the case
$d=1$ and $\lambda>0$, the operator $T_\lambda$ is simply
translation by $\lambda$ units to the right.
Given $1\leq p<\infty$,
$f\in L_p(\R)$, and $\Lambda\subset\R$, the resulting space
$X_p (f,\Lambda) \equiv \overline{\Span\{T_\lambda f\}}_{\lambda\in\Lambda}$ and set
$\{T_\lambda f\}_{\lambda\in\Lambda}$ have been studied in a variety
of contexts and in particular arise in the study of wavelets and
Gabor frames \cite[CDH]{HSWW}.

 Some
of the natural problems to consider when studying translations of a
fixed function $f$ relate to characterizing when can $X_p
(f,\Lambda)  =L_p(\R^d)$ and when can $\{T_\lambda
f\}_{\lambda\in\Lambda}$ be ordered to form a coordinate system such
as a (unconditional) Schauder basis or (unconditional)  Schauder
frame for $L_p(\R^d)$. For $d=1$, the cases when  $\Lambda=\Z$ or
$\Lambda=\N$ are of particular interest. For $1\leq p\leq 2$, a
Fourier transform argument yields that there does not exist an $f\in
L_p(\R)$ such that $ X_p (f,\Z) =L_p(\R)$ \cite{AO}. On the other
hand, for all $\{\lambda_n\}_{n\in\Z}\subset\R\setminus\Z$ such that
$\lim_{n\rightarrow\pm\infty}|\lambda_n-n|=0$, there exists $f\in
L_2(\R)$ such that $X_2 (f,\Z) = L_2(\R)$ \cite{O}. The case
$2<p<\infty$ is completely different, as for all $2<p<\infty$ there
exists $f\in L_p(\R)$ such that $X_p (f,\Z) = L_p(\R)$
 and, moreover,   $T_m f\not\in X_p (f,\Z\setminus \{m\})$ for all
 $m\in\Z$ \cite{AO}.

Suppose that $f\in L_p(\R)$ and that $\{T_\lambda
f:\lambda\in\Lambda\}$ is an unconditional basic sequence in
$L_p(\R)$. What can be said about $X_p(f,\Lambda)$? If $1\le p\le 2$
then $\{T_\lambda f :\lambda\in\Lambda\}$ must be equivalent to the
unit vector basis of $\ell_p$. If $2<p\le 4$, then $X_p (f,\Lambda)$
embeds into $\ell_p$ but $(T_\lambda f)_{\lambda\in\Lambda}$ need
not be equivalent to the unit vector basis of $\ell_p$.  These
results  were shown in  \cite{OSSZ}  and imply in particular that
for all $1\le p\le 4$ (with \cite{CDH} for the case $p=2$), there is
no function $f\in L_p(\mathbb R)$ and set $\Lambda\subset\mathbb R$
so that $(T_\lambda f:\lambda\in\Lambda)$ is an unconditional basis
for $L_p(\mathbb R)$.  For $4<p<\infty$, there exists $f\in L_p
(\R)$ and $\Lambda\subseteq \N$ such that $(T_\lambda
f)_{\lambda\in\Lambda}$ is an unconditional basic sequence and
$L_p(\R)$ embeds isomorphically into $X_p (f,\Lambda)$ \cite{OSSZ}.
In Section~3 we prove that if $(T_\lambda f)_{\lambda\in\Lambda}$ is
an unconditional basic sequence in $L_p(\R)$ with $2<p<\infty$ such
that $X_p(f,\Lambda)$ is complemented in $L_p(\R)$ then $(T_\lambda
f)_{\lambda\in\Lambda}$ must be equivalent to the unit vector basis
of $\ell_p$. In particular, $X_p (f,\Lambda) \ne L_p (\R)$.  Thus,
by filling the gap $(4,\infty)$, we have for all $1\le p<\infty$,
that there is no function $f\in L_p(\mathbb R)$ and set
$\Lambda\subset\mathbb R$ so that $(T_\lambda f:\lambda\in\Lambda)$
is an unconditional basis for $L_p(\mathbb R)$.

A basic sequence $(x_i)$ can uniquely represent every vector in its
closed span in terms of an infinite series.  We now drop the
uniqueness requirement of the representation and consider frames
formed by translating a single function.
 Though there exists $\Lambda\subset\R$ and $f\in L_2(\R)$ such
that $X_2 (f,\Lambda)  = L_2(\R)$, there does not exist
$\Lambda\subset\R$ and $f\in L_2(\R)$ such that $\{T_\lambda
f\}_{\lambda\in\Lambda}$ is a Hilbert frame for $L_2(\R)$\cite{CDH}.
In the case of translation only by natural numbers, for all
$\Lambda\subset\N$ and $f\in L_2(\R)$, the sequence $(T_\lambda
f)_{\lambda\in\Lambda}$ is a Hilbert frame for $X_2(f,\Lambda)$ if
and only if it is a Riesz basis for $X_2(f,\Lambda)$, i.e.,
$(T_\lambda f)$ must be equivalent to the unit vector basis of
$\ell_2$  \cite{CCK}. In Section~\ref{S3} we provide some background
on Schauder frames for Banach spaces and prove that there exists a
function $f\in L_p(\R)$ and sequence $(g^*_n)_{n\in\N}\subset
L_p^*(\R)$ such that $(T_{n} f, g^*_n)_{n\in\N}$ forms an
unconditional Schauder frame for $L_p(\R)$ if (and only if)
$2<p<\infty$.  More generally, we prove that for every $2<p<\infty$,
$d\in\N$ and unbounded sequence $(\lambda_n)_{n\in\N}$ in $\R^d$,
there exists a function $f\in L_p(\R^d)$ and sequence
$(g_n^*)_{n\in\N}$ such that $(T_{\lambda_n}f,g^*_n)_{n\in\N}$ forms
an unconditional Schauder frame for $L_p(\R^d)$.

For $2<p<\infty$, if $L_p$ embeds into $X_p(f,\Lambda)$ and
$X_p(f,\Lambda)$ is complemented in $L_p(\R)$ then
 $(T_\lambda f)_{\lambda\in\Lambda}$ cannot be an unconditional basic
sequence in $L_p(\R)$.  However, it is possible that  $(T_\lambda
f)_{\lambda\in\Lambda}$ can be blocked into an unconditional FDD.
 We prove in Section~\ref{S4}   that for $2<p<\infty$ there exists $f\in L_p
(\R)$ and $\Lambda \subseteq \N$ so that $X_p(f,\Lambda) $ is
isomorphic to $L_p$, $X_p(f,\Lambda) $ is complemented in $L_p$, and
$\{ T_\lambda f\}_{\lambda\in\Lambda}$ can be blocked to form an
unconditional finite dimensional decomposition (unconditional FDD)
for $X_p(f,\Lambda)$.

In Section~\ref{S5}, we study the restriction operator
$T_I:L_p\rightarrow L_p$ given by $x\mapsto x|_I$ where $I\subset\R$
is some bounded interval.  Assuming $(T_{\lambda_i}f)$ is an
unconditional basic sequence, we characterize for what values of
$1\leq p<\infty$ must the map $T_I:X_p (f,(\lambda_i))\rightarrow
L_p$ be compact for all bounded intervals $I\subset R$. We prove as
well other relationships between the restriction operator $T_I:X_p
(f,(\lambda_i))\rightarrow L_p$ and the structure of $X_p
(f,(\lambda_i))$. Lastly, in Section~\ref{S6} we give some open
problems.

\section{Unconditional bases of translates}\label{S2}

\begin{thm}\label{T:1}
Let  $2< p<\infty$ and $f\in L_p(\R)$. If $(T_\lambda f
:\lambda\in\Lambda)$ is an unconditional basis of  $X_p(f,\Lambda)$
and
 $X_p(f,\Lambda)$ is complemented in $L_p(\R)$, then
 $(T_\lambda f)_{\lambda\in\Lambda}$ is equivalent to the unit vector basis of $\ell_p$.
\end{thm}

We will need the following result from \cite{JO}.

\begin{prop}\label{P:2}
\cite[Section 3, Lemma 2]{JO}. Let $1\le q\le 2$. Let $(f_i)
\subseteq L_q (\R)$ be seminormalized and unconditional basic.
Assume that for some $\vp>0$ there exists a sequence of disjoint
measurable sets $(B_i)_{i=1}^\infty$ with $\|f_i|_{B_i}\|_q \ge
\vp$, for all $i$. Then $(f_i)_{i=1}^\infty$ is equivalent to the
unit vector basis of $\ell_q$.
\end{prop}

\begin{proof}[Proof of Theorem \ref{T:1}]
Without loss of generality we can assume that $\|T_{\lambda_i} f\|_p
= \|f\|_p =1$, for $i\in\N$. Put $f_i= T_{\lambda_i} f$ and
$X=X_p(f,\Lambda)$ and let $Y\subset L_p(\R)$ be a complement of $X$
in $L_p(\R)$.
 Denote the biorthogonals  of $(f_i)$ inside $X^*$ by $(\gb_i)$ and let $g_i$,  $i\in\N$, be the extension of $\gb_i$  to an element  of $L^*_p(\R)=L_q(\R)$ with $g_i|_Y\equiv 0$.
Thus $(g_i)$ is an unconditional basic sequence inside $L_q(\R)$
which is biorthogonal to $(f_i)$ and vanishes on $Y$.

Recall that if $\{ T_\lambda f:\lambda\in\Lambda\}$ is a basic
sequence in $L_p(\R)$, for some $f\in L_p (\R)$ and $\Lambda
\subseteq \R$, then $\Lambda$ is uniformly discrete \cite{OSSZ}.
That is, we may choose $\delta>0$ such that
 $$0<\delta<\inf\{|\lambda-\mu|: \lambda,\mu\in\Lambda, \lambda\not=\mu\}.$$
For $j\in\Z$, we define the interval $I_j=[j\delta,(j+1)\delta)$.

\medskip
\noindent{\bf Claim.}  There exist $N\in\N$ and $\vp>0$, so that:
for all $i\in\N$ there is an $l_i\in\Z$  and a $j_i\in\{ l_i,
l_i+1,\ldots ,l_i+N\}$, so that
$$   l_i\not=l_{i'}, \text{ if $i\not=i'$, and }  \big\| g_i|_{I_{j_i}}\big\|_q>\vp.$$
Indeed, choose first $l_0\in \Z$ and $N\in\N$ so that
$$\|  f |_{\R\setminus \bigcup_{j=l_0}^{l_0+N-1}I_j}\|_p ^p
= \int_{-\infty}^{l_0 \delta} \big|f(z)\big|^p dz
+\int_{(l_0+N)\delta}^\infty \big|f(z)\big|^pdz <  \frac1{2^p}
\Big(\sup_{i\in\N} \|g_i\|_q^p\Big)^{-1} $$ Then  for $i\in\N$
choose  $l_i\in\Z$ such that
$$l_0\delta\le (l_i+1)\delta-\lambda_i<(l_0+1)\delta.$$
Note that if $i\not=i'$ it follows that
$|\lambda_i-\lambda_{i'}|>\delta$ and, thus, $l_i\not=l_{i'}$.
Moreover,
\begin{align*}
 \|  f_i |_{\R\setminus \bigcup_{j=l_i}^{l_i+N}I_j}\|_p^p
 & =   \int_{-\infty}^{l_i \delta} \big|f(x-\lambda_i)\big|^p dx
 + \int_{(l_i+N+1)\delta}^\infty \big|f(x-\lambda_i)\big|^pdx\\
 & = \int_{-\infty}^{l_i \delta-\lambda_i} \big|f(z)\big|^p dz
 + \int_{(l_i+N+1)\delta-\lambda_i}^\infty \big|f(z)\big|^pdz\\
 &\le  \int_{-\infty}^{l_0 \delta} \big|f(z)\big|^p dz
 +\int_{(l_0+N)\delta}^\infty \big|f(z)\big|^p dz<\frac1{2^p}
 \Big(\sup_{i\in\N} \|g_i\|_q^p\Big)^{-1}.
\end{align*}
Thus, by H\"older's Theorem and the fact that $\|f\|_p =1$, it
follows that
\begin{align*}
\big\| g_i|_{ \bigcup_{j=l_i}^{l_i+N}I_j}\big\|_q \ge \int_{
\bigcup_{j=l_i}^{l_i+N}I_j } g_i f_i dz = 1- \int_{ \R\setminus
\bigcup_{j=l_i}^{l_i+N}I_j} g_i f_i \ge 1-\big\|g_i\big\|_q \big\|
f |_{ \R\setminus \bigcup_{j=l_i}^{l_i+N}I_j}\big\|_p \ge \frac12\ .
\end{align*}
Letting $\vp=\frac1{2(N+1)}$ we deduce our claim.

Since the $l_i$'s are   distinct, it follows that for each $k\in\Z$
$$\#\{i\in\N: j_i=k\}\le\# \{ i: k\in [l_i,l_i+N]\}= \#\{i: l_i\in[k-N,k]\} \le N+1\ .$$

For each $k\in\Z$ we can  order the (possibly empty) set $\{ i \in\N
: j_i = k \}$ into $i(k,1), i(k,2),\allowbreak \ldots i(k,m_k)$,
with $0\le m_k\le N+1$ (where we let $m_k=0 $  if $\{ i \in\N :
j_i=k\}$  is empty).

For $k\in\Z$ and $s\in\{1,2\ldots m_k\}$  put
$g^{(s)}_k=g_{i(k,s)}$.

For each  $s\le N+1$ it follows that the sequence $(g_k^{(s)}:k\in\Z
\text{ and } m_k\ge s) $, satisfies the condition of Proposition
\ref{P:2},  with $B_k=[k\delta, (k+1)\delta)$, as long it is
infinite and must therefore be equivalent to the unit vector basis
of $\ell_q$. Thus, since
 $\overline{\Span \{g_i\}}_{i\in\N}$
 is the unconditional sum of $[g_k^{(s)}:k\in\Z\text{ and } m_k\ge s]$,
 $s=1,2,\ldots N+1$, it follows that $(g_i)$ must be equivalent to the unit vector basis of  $\ell_q$. Since
 $X$ is complemented in $L_p(\R)$  and $g_i|_Y=0$  (recall that $L_p=X\oplus Y$) and $g_i|_X=\gb_i$ for $i\in\N$ it follows that $(\gb_i)$ is equivalent  to  $(g_i$) and, thus,  also
 equivalent to the unit vector basis of $\ell_q$. But this implies that $(f_i)$ is equivalent to the unit vector basis of $\ell_p$.
\end{proof}

\begin{cor}\label{C:3}
If $f\in L_p (\R)$, $1\le p<\infty$, $(T_\lambda
f)_{\lambda\in\Lambda}$ is an unconditional basis for $X_p
(f,\lambda)$ then $X_p(f,\Lambda) \ne L_p (\R)$.
\end{cor}

\section{Unconditional Schauder frames of translates}\label{S3}
In Section \ref{S2}, it was shown that there does not exist an
unconditional basis of translates of a single function for $L_p(\R)$
for any value $1\leq p<\infty$.  In contrast to this, we will show
that $L_p(\R)$ has an unconditional Schauder frame of integer
translates of a single function if (and only if) $2<p<\infty$.
Before proving this result, we will develop some basic theory of
Schauder frames.

If $X$ is a separable Banach space, then a sequence
$(x_i,g^*_i)_{i=1}^\infty\subset X\times X^*$ is called a {\em
Schauder frame} for $X$ if
\begin{equation}\label{frameDef}
x=\sum_{i=1}^\infty g^*_i(x) x_i\quad\text{ for all }\ x\in  X.
\end{equation}
A Schauder frame $(x_i,g^*_i)_{i=1}^\infty\subset X\times X^*$ is
called an {\em unconditional Schauder frame} for $X$ if the series
(\ref{frameDef}) converges unconditionally for all $x\in X$. Recall
that a series converges {\em unconditionally} if it converges for
any ordering of the elements of the series.

 Let $X$ be a separable Banach space.
 Assume that a sequence $(x_i,g^*_i)_{i=1}^\infty\subset X\times X^*$ satisfies that the
operator $S:X\rightarrow X$ defined by $S(x)=\sum_{i=1}^\infty
g^*_i(x) x_i$ is well defined (and hence bounded due to the uniform
boundedness principle). $S$ is called the {\em frame operator} for
$(x_i,g^*_i)_{i=1}^\infty$. Note that the sequence
$(x_i,g^*_i)_{i=1}^\infty\subset X\times X^*$ is a Schauder frame if
and only if the frame operator is the identity. We define
$(x_i,g^*_i)_{i=1}^\infty$ to be an {\em approximate Schauder frame}
if the frame operator is bounded, one to one, and onto (hence has
bounded inverse), and we define
  $(x_i,g^*_i)_{i=1}^\infty$ to be an {\em unconditional
approximate Schauder frame} if it is an approximate Schauder frame
and the series $\sum_{i=1}^\infty g^*_i(x) x_i$ converges
unconditionally for all $x\in X$.
A similar notion of a frame was studied by Thomas in the context of $\ell_\infty^n$ \cite{T}.

\begin{lem}\label{L:1}
Let $X$ be a separable Banach space and let
$(x_i,g^*_i)_{i=1}^\infty\subset X\times X^*$ be an approximate
Schauder frame for $X$ with frame operator $S$.
Then $(x_i,(S^{-1})^*g^*_i)_{i=1}^\infty$ is a Schauder frame for $X$.
Furthermore, if $(x_i,g^*_i)_{i=1}^\infty$ is an unconditional approximate
Schauder frame for $X$, then $(x_i,(S^{-1 })^*g^*_i)_{i=1}^\infty$
is an unconditional Schauder frame for $X$.
\end{lem}

\begin{proof}
Let $x\in X$.  We have that $S$ and $S^{-1}$ are bounded.
Thus,
$$x=S(S^{-1} x)= \sum_{i=1}^\infty g^*_i(S^{-1}x) x_i= \sum_{i=1}^\infty ((S^{-1})^*g^*_i)(x)
x_i.
$$
Hence, $(x_i,(S^{-1 })^*g^*_i)_{i=1}^\infty\subset X\times X^*$ is a
Schauder frame for $X$.
Assume that $(x_i,g^*_i)_{i=1}^\infty$ is an unconditional approximate Schauder
frame for $X$.  If $x\in X$ and $\pi:\N\rightarrow\N$ is a
permutation, then $S(y)=\sum_{i=1}^\infty g^*_{\pi(i)}(y) x_{\pi(i)}$ for all $y\in X$.
Hence,
$$x=S(S^{-1} x)= \sum_{i=1}^\infty g^*_{\pi(i)} (S^{-1}x) x_{\pi(i)}=
\sum_{i=1}^\infty \left( (S^{-1})^*g^*_{\pi(i)}\right)\!(x)\, x_{\pi(i)}\ .
$$
\end{proof}

In particular, Lemma \ref{L:1} implies that $L_p(\R^d)$ has a
(unconditional) Schauder frame formed by translating a single
function if and only if it has an (unconditional) approximate
Schauder frame formed by translating a single function.  This is
important for us, as we will provide an explicit construction for an
approximate Schauder frame of translates for $L_p(\R^d)$ and then
apply Lemma \ref{L:1} to obtain a Schauder frame of translates for
$L_p(\R^d)$ for any $p>2$. One way to verify that a sequence
$(x_i,g^*_i)_{i=1}^\infty\subset X\times X^*$ with frame operator
$S$ is an approximate Schauder frame is to show that $\|S-Id_X\|<1$.

\begin{thm}\label{TH1}
Let $2<p<\infty$ and $d\in\N$.  If $(\lambda_n)_{n\in\N}$ is an
unbounded sequence in $\R^d$ then there exists a function $f\in
L_p(\R^d)$ and a sequence $(g^*_n)_{n\in\N}\subset L_p^*(\R^d)$ such
that $(T_{\lambda_n} f, g^*_n)_{n\in\N}$ forms an unconditional
Schauder frame for $L_p(\R^d)$.
\end{thm}

\begin{proof}
Let $(e_i)_{i=1}^\infty$ be a normalized unconditional Schauder
basis for $L_p(\R^d)$ with biorthogonal functionals
$(e_i^*)_{i=1}^\infty$ such that $e_i\in L_p(\R^d)$ is a function
satisfying $\textrm{diam}(\supp(e_i))\leq 1$ for all $i\in\N$. Let
$C_u$ be the constant of unconditionality of $(e_i)_{i=1}^\infty$.
For each $k\in\N$, choose $N_k\in\N$ such that $(\sum_{k=1}^\infty
N_k^{1-p/2})^{1/p}<\frac{1}{2C_u}$. We now inductively construct
natural numbers $n_{1,1} < n_{1,2} < \cdots < n_{1,N_1} < n_{2,1} <
\cdots < n_{2,N_2} < \cdots$ such that if $(k,i)>(s,t)$, in the
lexicographic order, then
\begin{equation}\label{DisjointSupport00a}
|\lambda_{n_{k,i}}-\lambda_{n_{s,t}}|>1\ ,\text{ and}
\end{equation}
\begin{equation}\label{DisjointSupport0a}
\supp(T_{-\lambda_{n_{k,i}}} e_{k})\cap \supp(T_{-\lambda_{n_{s,t}}}
e_{s})=\emptyset\ ,
\end{equation}
and if $(k,i)>(s,t)$, $(k,i)\ge (k',i')$, $(k,i)\ge (s',t')$, $(s,t)\neq (s',t')$ and $(s',t')\neq (k',i')$ then
\begin{equation}\label{DisjointSupporta}
\supp(T_{\lambda_{n_{s,t}}-\lambda_{n_{k,i}}} e_{k})\cap
\supp(T_{\lambda_{n_{s',t'}}-\lambda_{n_{k',i'}}} e_{k'})=\emptyset\ ,
\end{equation}
and, finally, if $(k,i)>(s,t)$, $(k,i)\ge (k',i')$, $(k,i)>(s',t')$
and $(s',t')\neq (k',i')$ then
\begin{equation}\label{DisjointSupportb}
\supp(T_{\lambda_{n_{k,i}}-\lambda_{n_{s,t}}} e_{s})\cap
\supp(T_{\lambda_{n_{s',t'}}-\lambda_{n_{k',i'}}} e_{k'})=\emptyset.
\end{equation}
As we are choosing $n_{k,i}\in\N$, we can clearly
satisfy~\eqref{DisjointSupport00a} and~\eqref{DisjointSupport0a} by
making $n_{k,i}$ sufficiently large since each of the functions $e_j$
has compact support, and the sequence $(\lambda_n)_{n\in\N}$ is
unbounded. As we have $\textrm{diam}(\supp(e_i))\leq1$ for all
$i\in\N$, we will automatically satisfy~\eqref{DisjointSupporta} when
$(k,i)=(k',i')$. The (finitely many) remaining cases
of~\eqref{DisjointSupporta} and~\eqref{DisjointSupportb} can then be
satisfied by making $n_{k,i}$ larger still using again the assumptions
that the functions $e_j$ have compact support, and the sequence
$(\lambda_n)_{n\in\N}$ is unbounded.

As a result of the above inductive construction, we obtain the
following two conditions. For all
$(k,i)\neq(s,t)$,
\begin{equation}\label{DisjointSupport0}
\supp(T_{-\lambda_{n_{k,i}}} e_{k})\cap \supp(T_{-\lambda_{n_{s,t}}}
e_{s})=\emptyset.
\end{equation}
For all $(s,t),(s',t'),(k,i),(k',i')$, if $(s,t)\neq(s',t')$,
$(s,t)\neq(k,i)$ and $(s',t')\neq (k',i')$, then
\begin{equation}\label{DisjointSupport}
\supp(T_{\lambda_{n_{s,t}}-\lambda_{n_{k,i}}} e_{k})\cap
\supp(T_{\lambda_{n_{s',t'}}-\lambda_{n_{k',i'}}} e_{k'})=\emptyset.
\end{equation}

Set $f:=\sum_{k=1}^\infty\sum_{i=1}^{N_k} N_k^{-1/2}
T_{-\lambda_{n_{k,i}}}e_k$.  Our first step is to show that $f\in
L_p(\R^d)$.
\begin{align*}
\int|f|^p \,\text{d}\mu
& = \int \Big|\sum_{k=1}^\infty\sum_{i=1}^{N_k}
N_k^{-1/2} T_{-\lambda_{n_{k,i}}}e_k \Big|^p\,\text{d}\mu\\
&=\sum_{k=1}^\infty\sum_{i=1}^{N_k}
N_k^{-p/2}\int|e_k|^p\,\text{d}\mu\quad\quad\quad\textrm{ by }(\ref{DisjointSupport0})\\
&=\sum_{k=1}^\infty
N_k^{1-p/2}\quad\quad\quad\quad\textrm{ as }\|e_k\|=1\textrm{ for all } k\in\N\\
&<\frac{1}{2^p C_u^p}.
\end{align*}
Thus we have that $f\in L_p(\R^d)$.  For each $j\in\N$, we define
$g^*_j\in L_p^*(\R^d)$ by
$$g^*_j=\begin{cases} N_k^{-1/2} e^*_k \quad &\mbox{if } j=n_{k,i} \textrm{ for some $k\in\N$ and }1\leq i\leq N_k, \\
0 & \mbox{otherwise.}  \end{cases}
$$

We now show that $\sum_{i=1}^\infty g^*_i(h) T_{\lambda_i} f$
converges unconditionally for all $h\in L_p(\R^d)$.  In particular,
this would imply that the frame operator for $(T_{\lambda_n} f,g_n^*)_{n\in\N}$ would be well defined and bounded.
By Proposition~1.c.1 in \cite{LT}, to prove that a series
$\sum_{i\in\N} x_i$ in a Banach space $X$ converges unconditionally,
it is sufficient to prove that for all $\vp>0$ there exists $N\in\N$
such that for all finite subsets $A\subset\N$ with $\min(A)>N$,
$\|\sum_{i\in A} x_i\|<\vp$. Let $\vp>0$ and let $h\in L_p(\R^d)$
such that $\|h\|=1$. Choose $M\in\N$ such that $\sum_{i=M}^\infty
N^{1-p/2}_i<\vp$ and
 $\|\sum_{i=M}^\infty
e^*_i(h)e_i\|<\vp$.  Let $A\subset \N$ such
that $\min(A)\geq n_{M,1}$.
We now have the following estimate.
\begin{align*}
\Big\|\sum_{i\in A} g^*_i(h) T_{\lambda_i} f\Big\|
&=\Big\|\sum_{\substack{(s,t)\in\N^2\\ n_{s,t}\in A} } N_s^{-1/2}e^*_s(h) \sum_{k=1}^\infty \sum_{i=1}^{N_k} N_k^{-1/2}
T_{\lambda_{n_{s,t}}-\lambda_{n_{k,i}}}e_k \Big\|\\
\noalign{\vskip6pt}
&\leq \Big\|\sum_{\substack{ (s,t)\in\N^2\\
n_{s,t}\in A}} e^*_s(h) N_s^{-1}e_s\Big\|
+ \Big\|\sum_{\substack{(s,t)\in\N^2\\
n_{s,t}\in A}} N_s^{-1/2}e^*_s(h) \sum_{(k,i) \neq (s,t)} N_k^{-1/2}
T_{\lambda_{n_{s,t}}-\lambda_{n_{k,i}}}e_k \Big\|\\
\noalign{\vskip6pt}
& = \Big\|\sum_{\substack{ (s,t)\in\N^2\\  n_{s,t}\in A}}e^*_s(h) N_s^{-1}e_s \Big\|
+ \bigg( \sum_{\substack{(s,t)\in\N^2\\n_{s,t}\in A}} N_s^{-p/2}|e^*_s(h)|^p
\sum_{(k,i)\neq (s,t)} N_k^{-p/2} \bigg)^{1/p}\quad\text{ by }(\ref{DisjointSupport})\\
\noalign{\vskip6pt}
&\leq C_u \Big\|\sum_{s=M}^\infty\sum_{t=1}^{N_s}e^*_s(h) N_s^{-1}e_s\Big\|
+ C_u \bigg(\sum_{s=M}^\infty\sum_{t=1}^{N_s} N_s^{-p/2} \sum_{k=1}^\infty
\sum_{i=1}^{N_k} N_k^{-p/2}\bigg)^{1/p}\ ,\\
&\hskip2.5truein \textrm{ as }\min(A)\geq n_{M,1} \text{ and } |e_s^* (h)| \le C_u \\
\noalign{\vskip6pt}
&= C_u \Big\|\sum_{s=M}^\infty e^*_s(h) e_s\Big\|
+ C_u \bigg(\sum_{s=M}^\infty N_s^{1-p/2}\sum_{k=1}^\infty N_k^{1-p/2}\bigg)^{1/p}   < C_u\vp+C_u\vp^{1/p}\frac{1}{2 C_u}.
\end{align*}
Since $\vp>0$ was arbitrary,
the series $\sum_{i=1}^\infty g^*_i(h) T_i f$
converges unconditionally.

Let $S$ be the frame operator for $(T_{\lambda_n} f, g_n^*)_{n\in\N}$.
To show that $(T_{\lambda_n} f, g_n^*)_{n\in\N}$ forms an approximate unconditional
Schauder frame for $L_p(\R^d)$ it suffices to show  that
$\|S-Id_{L_p(\R^d)}\|<\frac{1}{2}$.
Let $h\in L_p(\R^d)$ such that $\|h\|=1$.
Choose $M\in\N$ such $\|\sum_{i=M+1}^\infty e^*_i(h) e_i\|< \frac{1}{8}$
and $\|\sum_{i=n_{M,N_M}+1}^\infty g_i^*(h)T_{\lambda_i} f\|<\frac{1}{8}$.
Then
\begin{align*}
\|h-S(h)\| &= \Big\|\sum_{i=1}^\infty
e^*_i(h)e_i-\sum_{i=1}^\infty g^*_i(h)T_{\lambda_i} f \Big \|\\
& <  \Big\|\sum_{i=1}^M e^*_i(h)e_i-\sum_{i=1}^{n_{M,N_M}}
g^*_i(h)T_{\lambda_i}  f\Big \| + \frac{1}{8}+\frac{1}{8}\\
& =  \Big\|\sum_{i=1}^M
e^*_i(h)e_i-\sum_{s=1}^M\sum_{t=1}^{N_s}
N_s^{-1/2}e^*_s(h)\sum_{k=1}^\infty\sum_{i=1}^{N_k} N_k^{-1/2}
T_{\lambda_{n_{s,t}}-\lambda_{n_{k,i}}}e_k \Big\| + \frac{1}{4}\\
 &\leq \Big\|\sum_{i=1}^M
e^*_i(h)e_i-\sum_{s=1}^M\sum_{t=1}^{N_s}
N_s^{-1}e^*_s(h)e_s\Big\| \\
\noalign{\vskip6pt}
&\qquad \qquad + \Big\|\sum_{s=1}^M\sum_{t=1}^{N_s} N_s^{-1/2}e^*_s(h)
 \sum_{(k,i)\neq (s,t)} N_k^{-1/2}
T_{\lambda_{n_{s,t}}-\lambda_{n_{k,i}}}e_k \Big\| + \frac{1}{4}\\
& =  \Big\|\sum_{s=1}^M\sum_{t=1}^{N_s} N_s^{-1/2}e^*_s(h)
 \sum_{(k,i)\neq (s,t)} N_k^{-1/2}
T_{\lambda_{n_{s,t}}-\lambda_{n_{k,i}}}e_k \Big\| + \frac{1}{4}\\
&= \bigg(\sum_{s=1}^M\sum_{t=1}^{N_s} N_s^{-p/2} |e^*_s(h)|^p
 \sum_{(k,i)\neq (s,t)} N_k^{-p/2}\bigg)^{1/p}
 + \frac{1}{4}\qquad\quad\text{ by }(\ref{DisjointSupport})\\
&\leq C_u \bigg(\sum_{s=1}^\infty \sum_{t=1}^{N_s} N_s^{-p/2}
 \sum_{k=1}^\infty\sum_{i=1}^{N_k} N_k^{-p/2} \bigg)^{1/p} + \frac{1}{4}\\
& = C_u \bigg(\sum_{s=1}^\infty N_s^{1-p/2}
 \sum_{k=1}^\infty N_k^{1-p/2}\bigg)^{1/p} + \frac{1}{4}\\
&< C_u\frac{1}{2C_u}\frac{1}{2C_u}+\frac{1}{4}<\frac{1}{2}.
\end{align*}
Thus $\|S-Id_{L_p(\R^d)}\|<\frac{1}{2}$ and hence $S$ is bounded and
has a bounded inverse.  This gives that $(T_{\lambda_n} f,
g_n^*)_{n\in\N}$ forms an approximate unconditional Schauder frame
for $L_p(\R^d)$, and hence $(T_{\lambda_n} f,
(S^{-1})^*g_n^*)_{n\in\N}$ forms an unconditional Schauder frame for
$L_p(\R^d)$ by Lemma \ref{L:1}.
\end{proof}

We now discuss some consequences of Theorem \ref{TH1}.
Given a Schauder frame $(x_i,f_i)_{i=1}^\infty\subset X\times X^*$, let
$H_n:X\rightarrow X$ be the operator $H_n(x)=\sum_{i\geq n}f_i(x) x_i$.
The frame $(x_i,f_i)_{i=1}^\infty$ is called {\em shrinking}
if $\|x^*\circ H_n\|\rightarrow0$ for all $x^*\in X^*$.
A Schauder frame $(x_i,f_i)_{i=1}^\infty\subset X\times X^*$ for a Banach space
$X$ is shrinking if and only if $(f_i,x_i)_{i=1}^\infty\subset
X^*\times X^{**}$ is a Schauder frame for $X^*$ \cite{CL}.
Furthermore, every unconditional Schauder frame for a reflexive
Banach space is shrinking \cite{CLS},\cite{L}.
Thus   the following corollary of Theorem \ref{TH1} ensues.

\begin{cor}\label{cor}
Let $1<q<2$ and $d\in\N$.  If $(\lambda_n)_{n\in\N}\subset\R^d$ is
unbounded then there exists a function $f^*\in L_q^*(\R^d)$ and
sequence $(g_n)_{n\in\N}\subset L_q(\R^d)$ such that
$(g_n,T_{\lambda_n} f^* )_{n\in\N}$ forms a Schauder frame for
$L_q(\R^d)$.
\end{cor}

Note that in Corollary \ref{cor}, the dual functionals
$(T_{\lambda_n} f^*)_{n\in\N}$ are translations of a single function
as opposed to the vectors $(g_n)_{n\in\N}$.

In \cite{CDOSZ}, it is proven that every Schauder frame has an
associated basis.  Essentially, that means that Schauder frames can
be considered as projections of bases onto complemented subspaces.
In \cite{BFL}, it is proven that every shrinking Schauder frame for
a reflexive Banach space has a shrinking and boundedly complete
associated basis, and it follows from the proof that if the frame is
unconditional then the basis will be unconditional as well. Thus,
the Schauder frame $(T_{\lambda_n} f, g^*_n)_{n\in\N}$ for
$L_p(\R^d)$ will have an unconditional, shrinking, and boundedly
complete associated basis.


\section{Unconditional FDDs of translates}\label{S4}

In Section \ref{S2}, it was shown that for all $f\in L_p(\R)$,
$1\leq p< \infty$, and $\Lambda\subset\R$, if $(T_\lambda
f)_{\lambda\in\Lambda}$ is an unconditional basic sequence and $X_p
(f,\Lambda)$ is complemented in $L_p(\R)$ then $(T_\lambda
f)_{\lambda\in\Lambda}$ is equivalent to the unit vector basis for
$\ell_p$.  Instead of considering when $(T_\lambda
f)_{\lambda\in\Lambda}$ is an unconditional basic sequence, we now
study the cases where $(T_\lambda f)_{\lambda\in\Lambda}$ can be
blocked into an unconditional FDD.  Given a Banach space $X$, recall
that a sequence of finite dimensional spaces
$(F_i)_{i=1}^\infty\subset X$ is called a {\em finite dimensional
decomposition} or {\em  FDD} for $X$ if for every $x\in X$ there
exists for all $i\in\N$ a unique $x_i\in F_i$ such that
$x=\sum_{i=1}^\infty x_i$.  An FDD is called unconditional if the
series $x=\sum_{i=1}^\infty x_i$ converges unconditionally for all
$x\in X$.

\begin{thm}\label{thm:3.1}
Let $2<p<\infty$. There exists $f\in L_p (\R)$ and a subsequence
$(n_i)_{i=1}^\infty$ of $\N$ so that for   $X =  X_p
(f,(-n_i)_{i=1}^\infty)$, \ \ i)~$X$ is isomorphic to $L_p(\R)$, ii)
$X$ is complemented in $L_p(\R)$, and iii)~there exists a partition
of $\N$ into successive intervals $(J_j)_{j=1}^\infty$ so that
setting $F_j = \Span \{T_{- n_i}f\}_{i\in J_j}$,
$(F_j)_{j=1}^\infty$ forms an unconditional FDD for $X$.
\end{thm}

\begin{proof}[Proof of Theorem~\ref{thm:3.1}]
Let $\vp >0$ and choose a subsequence $(N_i)_{i=1}^\infty$ of $\N$ so that
\begin{equation}\label{eq:pf-thm3.1-a}
\sum_{j=1}^\infty N_j^{1-\frac{p}2} < \infty \quad \text{ and }\quad
\sum_{j=1}^\infty N_j^{\frac1p - \frac12} <\vp\ .
\end{equation}
Of course the second condition implies the first but we state both as they will be used.

Let $(h_j^i)_{j=1}^\infty$ be the normalized Haar basis for $L_p [3^i,3^i+1]$ for $i \in \N$.
Partition $\N$ into successive intervals $J_1,J_2,\ldots$ so that $|J_j| = N_j$
for $j\in \N$.

Let
$$f = \sum_{j=1}^\infty \sum_{i\in J_j} \frac1{\sqrt{N_j}} \, h_j^i\quad \text{ and let }\
f_i = T_{-3^i}f\ .$$
The choice of $3^i$ above yields, as in Section~2, that for $i\in J_j$,
\begin{equation}\label{eq:pf-thm3.1-b}
f_i = \frac1{\sqrt{N_j}}\, h_j  + g_i
\end{equation}
where $(h_j)$ is the normalized Haar basis for $L_p[0,1]$ and moreover the functions
$(g_i)$ have disjoint supports in $\R$.
Indeed
$$\supp g_i = \bigcup_{\ell\ne i} [3^\ell - 3^i\, ,\, 3^\ell -3^i +1]\ .$$
$f\in L_p(\R)$ since
$$\|f\|_p^p = \sum_{j=1}^\infty \sum_{i\in J_j}
\Big( \frac1{\sqrt{N_j}}\Big)^p
= \sum_{j=1}^\infty N_j \Big( \frac1{\sqrt{N_j}}\Big)^p
= \sum_{j=1}^\infty N_j^{1-\frac{p}2} < \infty\text{ (by \eqref{eq:pf-thm3.1-a}).}$$

Set
$$\bar h_j = \frac1{\sqrt{N_j}} \sum_{i\in J_j} \frac1{\sqrt{N_j}} f_i
= h_j + \frac1{\sqrt{N_j}} \sum_{i\in J_j} g_i\
\text{  (by \eqref{eq:pf-thm3.1-b}).}$$
Then
\begin{gather*}
\|\bar h_j - h_j\|  =  \Big\| \frac1{\sqrt{N_j}} \sum_{i\in J_j} g_i \Big\|_p
= \frac1{\sqrt{N_j}} \bigg( \sum_{i\in J_j} \|g_i\|_p^p \bigg)^{1/p}
\le N_j^{\frac1p - \frac12} \|f\|_p\ ,\ \text{(since }\ \|g_i\|_p \le \|f\|_p).
\end{gather*}
By \eqref{eq:pf-thm3.1-a}, for $\vp$ sufficiently small, it follows that $(\bar h_j)$ is
equivalent to $(h_j)$ and so $L_p$ embeds into $X$.

Let $E_j = \Span \{f_i :i\in J_j\}$ and $F_j = \Span \{\, \{ g_i :i\in J_j\} \cup \{ h_j\}\, \}$.
Since the $g_i$'s are disjointly supported and $(h_i)$ is unconditional, it follows that
$(F_j)$ is an unconditional FDD for its closed linear span, $Y$.
$E_j$ is a co-dimension one subspace of $F_j$ and thus $(E_j)$ is an
unconditional FDD for its closed linear span, $X$.
$Y$ is isometric to $L_p [0,1] \oplus \ell_p$ and this in turn is isomorphic to $L_p(\R)$.

Since $X$ contains an isomorphic copy of $L_p$, to prove that $X$ is
isomorphic to $L_p$ it suffices to prove that $X$ is complemented in
$Y$. Indeed $Y$ is complemented in $L_p$ and so if $X$ is
complemented in $Y$ then $X$ is a complemented subspace of $L_p$
which contains a complemented copy of $L_p$ (by \cite{JMST}). By
Pe{\l}czy\'nski's decomposition method \cite{LT}, $X$ is isomorphic
to $L_p$.

To accomplish this we will first define certain projections $P_j$ from $F_j$ onto $E_j$
and then prove that $P = \sum_j P_j$ is a projection of $Y$ onto $X$.

Let $z_j = \sum_{i\in J_j} N_j^{-1/p} g_i \in F_j \setminus E_j$.
We will prove that the seminormalized sequence $(z_j)$ satisfies $d(z_j,E_j) \ge c >0$
for all $j$ and some $c$.
$P_j$ will then be the projection of $F_j$ onto $E_j$ that sends $z_j$ to $0$ and
hence the $P_j$'s will be uniformly bounded.

We let $(\tilde h_j)$ be the biorthogonal sequence to $h_j$ in $L_q [0,1]$ given by
$\tilde h_j = |h_j|^{p-1} \text{sign}(h_j)$.
Thus $\|\tilde h_j\|_q = 1$.
Set $\tilde g_j = \frac{|g_j|^{p-1} \text{sign}(g_j)}{\|g_j\|^p_p} \in L_q
[0,1]$.
Note that $\tilde g_j(g_j) =1$ and $\tilde g_j(g_i) =0$ for $i\ne j$.
Furthermore $\|\tilde g_j\|_q =\frac1{\|g_j\|_p}$  and $\supp (\tilde
g_j) = \supp (g_j)$.

Next let
$$\phi_j = N_j^{\frac12-\frac1q} \tilde h_j - N_j^{-1/q} \sum_{i\in J_j} \tilde g_i\ .$$
$F_j$ is isometric to $\ell_p^{N_j+1}$ and
$$\|\phi_j |_{F_j}\|
= \sup \left\{ \frac{\Big|N_j^{\frac12 -\frac1p} a_0 - N_j^{-1/q} \sum\limits_{i\in J_j} a_i \Big|}
{\Big( |a_0|^p + \sum\limits_{i\in J_j} |a_i|^p \|g_i\|_p^p \Big)^{1/p}} \right\}\ ,$$
the ``sup'' is taken over all nonzero $(a_i)_0^{N_j} \in \ell_p^{N_j+1}$
\begin{equation*}
\begin{split}
\|\phi_j |_{F_j}\|
= \sup \left\{
\frac{\Big|N_j^{\frac12-\frac1p} a_0
- \sum\limits_{i\in J_j} N_j^{-1/q} \|\tilde g_i\|_q\, a_i \|g_i\|_p \Big|}
{\Big( |a_0|^p + \sum\limits_{i\in J_j} |a_i|^p \|g_i\|_p^p \Big)^{1/p}} \right\} = \Big\| N_j^{\frac12 - \frac1q} e_0
+ \sum_{i\in J_j} N_j^{-1/q} \|\tilde g_i\|_q\, e_i\Big\|_{\ell_q^{N_j+1}}\,,
\end{split}
\end{equation*}
where $(e_i)_0^{N_j}$ is the unit vector basis for $\ell_q^{N_j+1}$.

Thus, since $\tilde h_j$ and $(\tilde g_i)_{i\in J_j}$ are disjointly supported in $L_q(\R)$
$$\| \phi_j |_{F_j}\|
= \Big\| N_j^{\frac12 -\frac1q} \, \tilde h_j + \sum_{i\in J_j} N_j^{-1/q} \tilde g_i\Big\|_{L_q}
= \|\phi_j\|_{L_q}\ .$$
Now for $i_0 \in J_j$,
\begin{equation*}
\phi_j (f_{i_0})
 = N_j^{\frac12 - \frac1q} \, \tilde h_j \Big( \frac1{\sqrt{N_j}}\, h_j\Big)
- N_j^{-1/q} \sum_{i\in J_j} \tilde g_i (g_{i_0})  = N_j^{-1/q} - N_j^{-1/q} = 0\ .
\end{equation*}

Thus $\text{Ker } \phi_j |_{F_j} = E_j$.
It follows that
$$d(z_j,E_j) \ge \frac{\phi_j (z_j)}{\|\phi_j\|_q} = \frac1{\|\phi_j\|_q} \ge c >0$$
for some $c$ and all $j$, since $(\phi_j)$ is seminormalized in $L_q$.

We define $P_j :F_j\to E_j$ by
\begin{equation}\label{eq:pf-thm3.1-c}
\lambda z_j + \sum_{i\in J_j} b_i \Big( \frac1{\sqrt{N_j}} \, h_j + g_i \Big)
\longmapsto \sum_{i\in J_j} b_i \Big( \frac1{\sqrt{N_j}}\, h_j + g_i\Big)\ ,
\end{equation}
and let $P = \sum_{j=1}^\infty P_j$.
Let $C = \sup_j \|P_j\|$.
Let $x = \sum_{j=1}^\infty x_j$, $x_j = a_j h_j + \sum_{i\in J_j} c_i g_i \in F_j$.
Then by  \eqref{eq:pf-thm3.1-c},
$$P(x) = \sum_{j=1}^\infty P_j (x_j)
= \sum_{j=1}^\infty \Big(a_j h_j +
  \sum_{i\in J_j} d_i g_i\Big) $$
for some sequence $(d_i)$.
Thus to show that $P$ is bounded we need only show that for some $K<\infty$,
$$\Big\| \sum_{j=1}^\infty \sum_{i\in J_j} d_i g_i\Big\|_p \le K\|x\|\ .$$
Now
\begin{equation*}
\begin{split}
\Big\| \sum_{j=1}^\infty \sum_{i\in J_j} d_i g_i\Big\|_p
& = \bigg( \sum_{j=1}^\infty \Big\| \sum_{i\in J_j} d_i g_i\Big\|^p \bigg)^{1/p}\\
\noalign{\vskip6pt}
& \le C\bigg( \sum_{j=1}^\infty \|x_j\|^p\bigg)^{1/p}
= C \bigg( \sum_{j=1}^\infty \bigg( |a_j|^p + \Big\| \sum_{i\in J_j} c_i g_i\Big\|^p\bigg)\,
\bigg)^{1/p}\ .
\end{split}
\end{equation*}
$(h_j)$ admits a lower $\ell_p$ estimate since $p>2$.
Thus for some $\bar K$,
\begin{equation*}
\Big\| \sum_{j=1}^\infty \sum_{i\in J_j} d_i g_i\Big\|_p
 \le C\bigg[ \bar K^p \Big\|\sum_{j=1}^\infty a_j h_j\Big\|_p^p
+ \sum_{j=1}^\infty \Big\| \sum_{i\in J_j} c_i g_i\Big\|^p \bigg]^{1/p}  \le C\, \bar K \|x\|\ .
\end{equation*}
\end{proof}

\section{Compactness of restriction operators}\label{S5}

In the case where $(T_{\lambda_i} f)_{i=1}^\infty$ is an unconditional basic sequence
of translates of some $f\in L_p (\R)$, $1\le p\le 2$, the space $X_p (f,(\lambda_i))$
must be quite thin as the next proposition reveals.

\begin{prop}\label{prop:3.2}
Let $f\in L_p(\R)$, $1\le p\le 2$.
Let $T_{\lambda_i} f \equiv f_i$ be such that $(f_i)$ is unconditional basic.
Let $I\subseteq \R$ be a bounded interval and $X =  X_p (f,(\lambda_i))$.
Then the map $T_I : X \to L_p (I)$, $x\longmapsto x|_I$, is a compact operator.
\end{prop}

\begin{proof}
For $p=1$ this follows by the proof of Corollary~2.4 \cite{OSSZ}.
In fact this holds under the assumption that $(f_i)$ is basic (and even less).

Suppose that $1<p\le 2$ and $\vp >0$.
Since $\sum_{i=1}^\infty \|f_i|_I \|_p^p <\infty$ (see Proposition~2.1, \cite{OSSZ})
there exists $N\in \N$ so that $(\sum_{i=N}^\infty \|f_i|_I \|^p)^{1/p} <\vp$.
Let $x = \sum_{i=N}^\infty a_i f_i$, $\|x\|_p =1$.
Then
$$\|x|_I\| \le \sum_{i=N}^\infty |a_i|\, \|f_i|_I\|
\le \bigg(\sum_{i=N}^\infty |a_i|^q\bigg)^{1/q}
\bigg(\sum_{i=N}^\infty \|f_i|_I \|_p^p \bigg)^{1/p}$$
by H\"older's inequality $(\frac1p +\frac1q =1)$.
Since $q\ge 2$, $(\sum_{i=N}^\infty |a_i|^q)^{1/q} \le (\sum_{i=N}^\infty |a_i|^2)^{1/2}$.
Furthermore, by the unconditionality of $(f_i)$, there exists a constant $K$ so that
$$\bigg(\sum_{i=1}^\infty |a_i|^2 \bigg)^{1/2} \le K\|x\| = K\ .$$
$K$ depends only on $p$, the unconditionality constant of $(f_i)$ and
$\|f\| = \|f_i\|$ for $i\in \N$.
Thus $\|x|_I \| \le K\vp$.
This proves that $T_I$ is a compact operator on $X$.
\end{proof}

We will show in Proposition~\ref{P:6} below that
Proposition~\ref{prop:3.2} fails for $p>2$. However, in the range
$2<p\leq 4$ we have the following result whose proof can be extracted
from the proof of~\cite[Theorem~2.11]{OSSZ}.
\begin{prop}\label{prop:3.2a}
Let $f\in L_p(\R)$, $2<p\le 4$.
Let $T_{\lambda_i} f \equiv f_i$ be such that $(f_i)$ is unconditional basic.
Then there is a basic sequence $(g_i)$ in $L_p(\R)$ equivalent to $(f_i)$ such that with
$Y=\overline{\Span}\{g_i:\,i\in\N\}$ for any bounded interval
$I\subseteq \R$ the map $T_I : Y \to L_p (I)$,
$y\longmapsto y|_I$, is a compact operator.
\end{prop}

\begin{proof}
Let $(h_j)$ be the normalized Haar basis for $L_p [0,1]$. For $i\in\Z$
and $j\in \N$ let $h_i^j$ be $h_j$ translated to $[i,i+1]$. Thus
$(h_i^j)$ is a normalized unconditional basis of $L_p(\R)$.

By approximating each $f_i$ by a simple dyadic function we find a
seminormalized block basis $(g_i)$  of $(h_i^j)$ such that
\begin{equation}
\label{eq:perturb}
\sum_{i=1}^\infty \big\| \, |f_i| - |g_i |\,\big\|_p < \infty\ .
\end{equation}
By a very useful observation of Schechtman~\cite{S} it follows that $(f_i)$ is
equivalent to $(g_i)$.

Set $Y=\overline{\Span}\{g_i:\,i\in\N\}$ and let $I$ be a bounded
interval. To show that $T_I : Y\to L_p(I)$ is compact we can assume
that $I=[-M,M]$ for some $M\in \N$. It follows from~\eqref{eq:perturb}
and~\cite[Proposition~2.1]{OSSZ} that $\sum_{i=1}^\infty \|g_i|_I \|_p^p
<\infty$. Fix $\vp >0$ and choose $N$ with
\begin{equation}
\label{eq:small-tail}
\sum_{i=N}^\infty \|g_i|_I \|_p^p < \vp\ .
\end{equation}
We note that $(g_i|_I)$ is a block basis of $(h_i^j)_{(j\in\N,\ -M\le
  i< M)}$ (after omitting zero vectors), and thus it is
unconditional basic. Let $y=\sum_{i=N}^\infty a_ig_i\in Y$. Recalling that
seminormalized unconditional basic sequences in $L_p(\R)$ satisfy
lower $\ell_p$ and upper $\ell_2$ estimates, we obtain the following
inequalities with some constant $C$ (dependent only on $p$ and the
norm of $f$).
\begin{eqnarray*}
  \| y|_I \|_p &=& \Big\| \sum_{i=N}^\infty a_i g_i|_I \|_p \le C
  \bigg( \sum_{i=N}^\infty |a_i|^2 \|g_i|_I \|_p^2\bigg)^{1/2} \\
  & \le & C \|(a_i)_{i=N}^\infty \|_{\ell_p}
\bigg( \sum_{i=N}^\infty \|g_i |_I \|_p^{\frac{2p}{p-2}}
\bigg)^{\frac{p-2}{2p}} \hfill \text{(using H\"older's inequality with
  $\textstyle \frac{p}2$
  and $\textstyle \frac{p}{p-2}$)} \\
&\le& C^2 \| y \|_p
\bigg(\sum_{i=N}^\infty \|g_i |_I \|_p^p\bigg)^{1/p}\leq C^2\vp^{1/p}
\| y \|_p \qquad \text{(using
  $\frac{2p}{p-2}\ge p$)}\ .
\end{eqnarray*}
This completes the proof.
\end{proof}

It is worth noting that  when the operators $T_I$ on some subspace
$X\subset L_p(\R)$  are  compact for all bounded  intervals $I$
then $X$ must embed into $\ell_p$ in a natural way as  the next
proposition reveals.
This observation and Proposition~\ref{prop:3.2} and~\ref{prop:3.2a}
above simplify some arguments in~\cite{OSSZ}.

If $P$ is a partition of $\R$ into bounded intervals $(I_j)$ we let $\E_P$ denote the {\em  conditional expectation
 operator on} $L_p(\R)$ given by
 $$\E_P(f)=\sum_{k=1}^\infty \int_{I_k} f(\xi)d\xi \frac{\chi_{I_k}}{m(I_k)}.$$

\begin{prop}\label{prop:3.3} Let $X$ be a subspace of $L_p(\R)$, $1\le p<\infty$. If for all
bounded intervals $I\subset \R$ the operator
$$T_I: X\to L_p(I), x\mapsto x|_I$$
is compact, then for all $\vp>0$ there exist a partition  $P$ of $\R$ into bounded intervals so that
for all $x\in S_X$, $\|x-\E_P(x)\|<\vp$.
Thus $X$ embeds into $\ell_p$.
\end{prop}
\begin{proof}
For $n\in\N$ let $Q_n$ be the set of dyadic intervals of length $2^{-n}$ in $[0,1)$, i.e.
$$Q_n=\big\{ [0,2^{-n}),[2^{-n}, 2^{1-n}),\ldots [1-2^{-n}, 1)\big\}. $$
 Then $\E_{Q_n}$ converges pointwise
to the identity on $L_p[0,1]$  and therefore
 there exists for every relatively compact set $K\subset L_p[0,1)$  and every $\delta>0$
 a large enough $k\in\N$ so that  for all $x\in K$, $\|x-\E_{Q_k}(x)\|<\vp$.
Choose a sequence $(\vp_n)\subset (0,1)$, with $\sum \vp_n<\vp$ and for
each $n$   choose a  dyadic  partition $P_n$ of the interval $[n,n+1)$ so that
for all $x\in S_X$, $\| x|_{[n,n+1)} -\E_{P_n}(x|_{[n,n+1)})\|\le \vp_n$.

By taking  $P$ to be the union of all $P_n$ we deduce our claim.
\end{proof}

Proposition~\ref{prop:3.2} fails in the case $2<p\le 4$, and of course for $p>4$ as well,
as shown by the next proposition.

\begin{prop}\label{P:6}
Let $2<p<\infty$.
There exists $f\in L_p(\R)$ and $(\lambda_i)_{i=1}^\infty \subseteq \N$ so that for
$f_i = T_{-\lambda_i}f$, $(f_i)_{i=1}^\infty$ is equivalent to the unit vector basis of $\ell_p$.
Moreover, letting $I= [0,1]$ and $T_I :X_p (f,(-\lambda_i)) \to L_p(I)$,
$x\mapsto x|_I$, $T_I$ is not a compact operator.
\end{prop}

\begin{proof}
Let $\frac1p +\frac1q =1$ and let $(N_j)_{j=1}^\infty$ be a subsequence of $\N$
satisfying $\sum_{j=1}^\infty N_j^{q-p} < \infty$.
Set $m_j =\lceil N_j^q\rceil$ for $j\in \N$ and let $(x_j)_{j=1}^\infty$ be a normalized sequence
of disjointly supported elements in $L_p(I)$.
Let $(J_j)_{j=1}^\infty$ be a partition of $\N$ into successive intervals so that
$|J_j| = m_j$ for all $j$.

For $i\in J_j$, let $x_i^j$ be $x_j$ placed on the interval $[3^i,3^i+1]$ by right translation
of $3^i$ units.
Define
$$f = \sum_{j=1}^\infty \sum_{i\in J_j} \frac1{N_j} \, x_i^j\ .$$
Note that
$$\|f\|_p^p = \Big\| \sum_{j=1}^\infty \sum_{i\in J_j}  \frac1{N_j} x_i^j\Big\|_p^p
= \sum_{j=1}^\infty m_j \frac1{N_j^p}
\le 2\sum_{j=1}^\infty \frac{N_j^q}{N_j^p} < \infty$$
so $f\in L_p(\R)$.

Setting $f_i = T_{-3^i} f$, for $i\in \N$, we have, as in the proof of Theorem~\ref{thm:3.1},
\begin{equation}\label{eq:P6}
f_i = \frac1{N_j} \, x_j +g_i\ ,\quad \text{for }\ i\in J_j\ ,
\end{equation}
where the $g_i$'s are disjointly supported, seminormalized and with supports disjoint
from $I$.
$(g_i)$ is thus equivalent to the unit vector basis of $\ell_p$.

To see that $(f_i)$ is equivalent to the unit vector basis of $\ell_p$, by \eqref{eq:P6}
it is sufficient to prove that for all $(a_i)_{i=1}^\infty \in \ell_p$,
\begin{equation}\label{eq:P6b}
\Big\| \sum_{i=1}^\infty a_i f_i\big|_I \Big\|_p
\le 2 \bigg( \sum_{i=1}^\infty |a_i|^p\bigg)^{1/p}\ .
\end{equation}

First note that for $j\in \N$,
$$\frac1{N_j} \Big| \sum_{i\in J_j} a_i\Big|
\le \frac1{N_j} \bigg( \sum_{i\in J_j} |a_i|^p\bigg)^{1/p} m_j^{1/q}
\le2 \bigg(\sum_{i\in J_j} |a_i|^p \bigg)^{1/p}\ .$$
Hence
$$\Big\| \sum_{i=1}^\infty a_i f_i\big|_I\Big\|_p^p
= \Big\| \sum_{j=1}^\infty \sum_{i\in J_j} a_i\, \frac1{N_j}\, x_j\Big\|_p^p
= \sum_{j=1}^\infty \Big| \sum_{i\in J_j} a_i \frac1{N_j}\Big|^p
\le 2 \sum_{i=1}^\infty |a_i|^p\ ,$$
which proves \eqref{eq:P6b}.

To see that $T_I$ is not compact, define $y_j = \sum_{i\in J_j} f_i$.
$\|y_j\|$ is of the order $m_j^{1/p}$ and $\| y_j |_I\| = \| \sum_{i\in J_j} \frac1{N_j} x_j\|
= \frac{m_j}{N_j} \ge m_j^{1/p}$.
Thus $m_j^{-1/p} y_j$ is seminormalized and weakly null in $L_p(\R)$, but
$\|T_I m_j^{-1/p} y_j\|_p \ge 1$ for all $j$.
\end{proof}

Using much the same argument we have

\begin{prop}\label{P:7}
Let $2<p<\infty$.
There exists $f\in L_p(\R)$ and translations of $f$, $f_i = T_{-3^i} f$, $i\in \N$, so that
\begin{itemize}
\item[i)] $(f_i)$ is basic,
\item[ii)] $L_p(\R)$ embeds isomorphically into $X_p (f,(-3^i))$,
\item[iii)] $(f_i)$ can be blocked into an unconditional {\rm FDD}.
\end{itemize}
\end{prop}

\begin{proof}[Sketch]
Let $(h_j)$ be the normalized Haar basis for $L_p [0,1]$.
For $i,j\in \N$, let $h_i^j$ be $h_j$ translated to $[3^i,3^i+1]$.
Set $f= \sum_j \sum_{i\in J_j} \frac1{N_j} h_i^j$ where $|J_j| = m_j \equiv \lceil N_j^q\rceil$
and $(J_j)$ partitions $\N$ into successive intervals.
As above $f_i = \frac1{N_j} h_j + g_i$, for $i\in J_j$, where $(g_i)$ is seminormalized
and disjointly supported in $\R\setminus [0,1]$.

If $y_j = \sum_{i\in J_j} f_i$, it follows that
$m_j^{-1/p} y_j = h_j + e_j$ where $(e_j)$
is seminormalized and disjointly supported in $\R\setminus [0,1]$.
Since $(h_i)$ admits a lower $\ell_p$-estimate, it follows that $(h_j + e_j)$ is equivalent
to $(h_j)$, proving ii).

Set $F_j = \Span \{f_i :i\in J_j\}$ and note that $F_j \subseteq
\overline{F}_j=\Span\{h_j ,(g_i)_{i\in
  J_j}\}$.
Since $(\,\overline{F}_j)$ is an unconditional FDD, so is $(F_j)$.

To see that $(f_i)$ is basic we need only note that $(f_i)_{i\in J_j}$ is uniformly
equivalent, over $j$, to the unit vector basis of $\ell_p^{m_j}$, as demonstrated
in the proof of Proposition~\ref{P:6}.~\end{proof}
\medskip

\section{Open problems}\label{S6}

We end with a collection of remaining open problems.

\begin{probs}\label{probs3:8}
Let $f\in L_p (\R)$ and let $(f_i)$ be a sequence of translates of $f$.
\begin{itemize}
\item[i)] For $1<p<\infty$, can $(f_i)$ ever be a basis of $L_p(\R)$?
\item[ii)] For $1<p<2$, can $(f_i)$ ever be basic such that $L_p$ embeds into
$\overline{\Span (f_i)}$?
\item[iiI)] For $1<p<\infty$, can $(f_i)$ ever be blocked into an
(unconditional) FDD for $L_p(\R)$?

\end{itemize}
\end{probs}

\vskip.2in
 \linespread{1.2}

\end{document}